\documentclass[11pt,bezier]{article}
\usepackage{amsmath,amssymb,amsfonts,euscript,graphicx}

\newtheorem{preproof}{{\bf \indent Proof.}}

\newenvironment{proof}[1]{\begin{preproof}{\rm
               #1}\hfill{$\Box$}}{\end{preproof}}

\newtheorem{cor}{\bf\indent Corollary}[section]
\newtheorem{example}{\bf\indent Example}[section]
\newtheorem{thm}{{\bf\indent Theorem}}[section]

\newtheorem{lem}{\bf\indent Lemma}[section]

\title{\bf \large  Inclusion graph of annihilators in a commutative ring}

\author{{\normalsize  {\sc Hana Safari$^{a}$, Farzad Shaveisi$^{a,}$\footnote{Corresponding author.}, Reza Nikandish$^{b}$}
}\vspace{3mm}\\
{\footnotesize{${}^{\mathsf{a}}$\it Department of Mathematics, Faculty of Science, Razi
University, Kermanshah, Iran}}\\
{\footnotesize{${}^{\mathsf{b}}$\it Department of Mathematics,
Jundi-Shapur University of Technology,  Dezful,
Iran}}\\
{\footnotesize{$\mathsf{h.safari@razi.ac.ir, f.shaveisi@razi.ac.ir, r.nikandish@ipm.ir}$}}}
\date{}

\begin{document}

\maketitle

%%%%%%%%%%%%%%%%%%%%%%%%%%%%%%%%%%%%%%%%%%%%%%%%%%%%%%%%%%%%%%%%%%%%%%%%%%%%%%%%%%%%%%%%%%%%%%%%%%%%%%%%%%%%%%%%%%%%%%%%%%%%%%%%%%%%%%%%%%%%%%%%%%%%%%%
\begin{abstract}
{\small Let $R$ be a commutative ring with identity, and let $Z(R)$ be the set of zero-divisors of $R$. The inclusion graph of annihilators in $R$, denoted by $\Gamma^{^{\prime}}(R)$, is a graph with the vertex set $Z(R)^*=Z(R)\setminus\{0\}$ and two distinct vertices $x$ and $y$ are adjacent if and only if $ann_R(x)\subseteq ann_R(y)$ or $ann_R(y)\subseteq ann_R(x)$.
It is proved that $\Gamma^{^{\prime}}(R)$ is not connected if and only if $R$ is reduced with $|\mathrm {Min}(R)|=2$.
Also, we show that if $\Gamma^{^{\prime}}(R)$ is a connected graph, then  the diameter of $\Gamma^{^{\prime}}(R)$ is at most $4$ and the girth of $\Gamma^{^{\prime}}(R)$ is at most $6$, if it contains a cycle. Moreover, we study the affinity between inclusion graph of annihilators  and complement of the annihilator graph (a well-known graph with the same vertices and two distinct vertices $x$ and $y$ are adjacent if and only if $ann_R(xy)\neq ann_R(x)\cup ann_R(y)$) associated with a commutative ring. Finally, we characterize all rings whose inclusion graphs of annihilators are complete.}
\end{abstract}
%{\bf Key words:} Inclusion graph, Annihilator, Connected graph, Diameter, Girth.}
%%%%%%%%%%%%%%%%%%%%%%%%%%%%%%%%%%%%%%%%%%%%%%%%%%%%%%%%%%%%%%%%%%%%%%%%%%%%%%%%%%%%%%%%%%%%%%%%%%%%%%%%%%%%%%%%%%%%%%%%%%%%%%%%%%%%%%%%%%%%%%%%%%%%%%%
\begin{center}\section{Introduction}\end{center}

Many interesting algebraic and combinatorial problems arise when we associate a
combinatorics object with an algebraic structure. Therefore, one of the most popular
and active areas in algebraic combinatorics is the study of graphs associated with
rings. Papers in this field apply combinatorial methods to obtain algebraic results
in ring theory (see for instance \cite{akbarii, anderson, Badawi, ab, hok,can}).
\par
Throughout this paper, $R$ is a commutative ring with unity unless otherwise specified. We denote by $\mathrm{Min}(R)$, $\mathrm {Max}(R)$, $\mathrm {Ass}(R)$, $Z(R)$, $\mathrm{Nil}(R)$ and $\mathrm{U}(R)$, the set of all minimal prime ideals of $R$, the set of all maximal ideals of $R$, the set of all associated prime ideals of $R$, the set of all zero-divisor elements of $R$, the set of all nilpotent elements of $R$ and the set of all invertible elements of $R$, respectively. The ring $R$ is said to be \textit{reduced} if it has no non-zero
nilpotent element. A nonzero ideal $I$ of $R$ is called \textit{essential} if $I$ has a non-zero intersection with any nonzero ideal of $R$. For every subset $A$ of $R$, we denote the  annihilator of $A$ by $ann_R(A)$. Moreover, for the subset $A$ of $R$, we let $A^*=A\setminus\{0\}$. For any undefined notation or terminology in ring theory, we refer the reader to \cite{ati, sharp}.
\par
  Let $G=(V,E)$ be a graph, where $V=V(G)$ is the set of vertices and $E=E(G)$ is the set of edges.  By $\overline{G}$, we mean the complement graph of $G$. The  diameter and the girth  of a graph $G$ are denoted by $\mathrm {diam}(G)$ and $\mathrm {gr}(G)$, respectively. We write $u-v$, to denote an edge with ends $u,v$. A graph $H=(V_0,E_0)$ is called a \textit{subgraph of} $G$ if $V_0\subseteq V$ and $E_0 \subseteq E$. Moreover, $H$ is called an \textit{induced subgraph by} $V_0$,  denoted by $G[V_0]$, if $V_0\subseteq V$ and $E_0=\{\{u,v\}\in E\, |\,u,v\in V_0\}$. Also $G$ is called a \textit{null graph} if it has no edge. A complete bipartite graph of part sizes $m, n$ is denoted by $K_{m,n}$. If  $m=1$, then the complete bipartite graph is called a \textit{star graph}. Also, a complete graph and a cycle  with $n$ vertices are denoted by $K_n$ and $C_n$, respectively. For any undefined notation or terminology in graph theory, we refer the reader to \cite{west}.
  \par
Let $R$ be a commutative ring with $1\neq 0$. The concept of  annihilator graph of a commutative ring $R$, denoted by $AG(R)$, was first introduced and studied by Badawi in \cite{Badawi}. The \textit{annihilator graph} $AG(R)$ is an undirected graph whose vertex set is $Z(R)^*$ and two distinct vertices $x$ and $y$ are adjacent if and only if ${\rm ann}_R(xy)\neq {\rm ann}_R(x)\cup {\rm ann}_R(y)$. More results on annihilator graph of $R$ may be found in \cite{af, Coloring, nikm}. This graph has inspired a significant body of subsequent work,
including studies on its metric dimension and its complements \cite{13,4}. On the other hand, the notion of inclusion ideal graph of a ring $R$ (not necessarily commutative) was first introduced in \cite{akm}. Let $R$ be a ring with unity (not necessarily commutative). The \textit{inclusion ideal graph of $R$}, denoted by $In(R)$,
is a graph whose vertices are all nontrivial left ideals of $R$ and two distinct left ideals
$I$ and $J$ are adjacent if and only if $I \subseteq J$ or $J \subseteq I$. Motivated by these papers, we introduce and study another kind of inclusion graphs, called inclusion graph of annihilators in a commutative ring, which contains the complement of the annihilator graph as a subgraph. The \textit{inclusion graph of annihilators} of a commutative ring  $R$ is a graph $\Gamma^{^{\prime}}(R)$ with the vertex set $Z(R)^*$ and two distinct vertices $x$ and $y$ are adjacent if and only if either ${\rm ann}_R(x)\subseteq {\rm ann}_R(y)$ or ${\rm ann}_R(y)\subseteq {\rm ann}_R(x)$.
In Section 2, we study the connectedness, diameter and girth of $\Gamma^{^{\prime}}(R)$.  In Section 3, we give conditions under which $\Gamma^{^{\prime}}(R)$  and $\overline{AG(R)}$ are identical. In Section 4, we study complete inclusion graphs of annihilators.
%%%%%%%%%%%%%%%%%%%%%%%%%%%%%%%%%%%%%%%%%%%%%%%%%%%%%%%%%%%%%%%%%%%%%%%%%%%%%%%%%%%%%%%%%%%%%%%%%%%%%%%%%%%%%%%%%%%%%%%%%%%%%%%%%%%%%%%%%%%%%%%%%%%%%%%
%%%%%%%%%%%%%%%%%%%%%%%%%%%%%%%%%%%%%%%%%%%%%%%%%%%%%%%%%%%%%%%%%%%%%%%%%%%%%%%%%%%%%%%%%%%%%%%%%%%%%%%%%%%%%%%%%%%%%%%%%%%%%%%%%%%%%%%%%%%%%%%%%%%%%%%
%%%%%%%%%%%%%%%%%%%%%%%%%%%%%%%%%%%%%%%%%%%%%%%%%%%%%%%%%%%%%%%%%%%%%%%%%%%%%%%%%%%%%%%%%%%%%%%%%%%%%%%%%%%%%%%%%%%%%%%%%%%%%%%%%%%%%%%%%%%%%%%%%%%%%
{\section{Connectedness, diameter and girth of $\Gamma^{^{\prime}}(R)$}}\vspace{-2mm}
In this section, basic properties of $\Gamma^{^{\prime}}(R)$ are studied. It is proved that $\Gamma^{^{\prime}}(R)$ is not connected if and only if $R$ is reduced with $|\mathrm {Min}(R)|=2$. Also, we show that gr$(\Gamma^{^{\prime}}(R))\in \{3,6,\infty \} $ and $\mathrm{diam}(\Gamma^{^{\prime}}(R))\leq 4$, if $\Gamma^{^{\prime}}(R)$ is a connected graph. Furthermore, we classify all rings $R$ in terms of gr$(\Gamma^{^{\prime}}(R))$.

We start with the following result which investigates the connectedness and diameter of $(\Gamma^{^{\prime}}(R))$.
 \begin{thm}\label{diam}
  Let $R$ be a ring. Then  $\Gamma^{^{\prime}}(R)$ is not connected if and only if $R$ is reduced with $|\mathrm {Min}(R)|=2$. Moreover, $\mathrm{diam}(\Gamma^{^{\prime}}(R))\leq 4$, if $\Gamma^{^{\prime}}(R)$ is connected.
\end{thm}
\begin{proof}
{ Suppose that $x,y$ are non-adjacent vertices of $\Gamma^{^{\prime}}(R)$.
  First assume that ${\rm ann}_R(x)\cup {\rm ann}_R(y)\neq Z(R)$ and let $z\in Z(R)\setminus {\rm ann}_R(x)\cup {\rm ann}_R(y)$. It is easily checked that  $x-xz-z-yz-y$ is a path connecting $x$ and $y$. Hence, we assume  ${\rm ann}_R(x)\cup {\rm ann}_R(y)= Z(R)$ and continue the proof in  the following cases:

\textbf{Case 1.} $|\mathrm {Min}(R)|\geq 3$.  Since ${\rm ann}_R(x)\cup {\rm ann}_R(y)= Z(R)$,  \cite[Theorem 2.1]{Huckaba} and prime avoidance Theorem (see \cite[Theorem 3.61]{sharp}) imply that $\mathfrak{p} \subseteq {\rm ann}_R(x)$ or $\mathfrak{p} \subseteq {\rm ann}_R(y)$, for every $\mathfrak{p} \in \mathrm {Min}(R)$. Without loss of generality, assume that $\mathfrak{p}_1,\mathfrak{p}_2\subseteq {\rm ann}_R(x)$, for some  $\mathfrak{p}_1,\mathfrak{p}_2 \in \mathrm {Min}(R)$. If $\mathfrak{p}_1 \subseteq {\rm ann}_R(y)$ or $\mathfrak{p}_2\subseteq {\rm ann}_R(y)$, then since ${\rm ann}_R(x)\nsubseteq {\rm ann}_R(y)$, we deduce that  $x-z-y$ is  a path from $x$ to $y$, for  every  $z\in {\rm ann}_R(x)\setminus {\rm ann}_R(y)$. If  $\mathfrak{p}_1 \nsubseteq {\rm ann}_R(y)$ and $\mathfrak{p}_2\nsubseteq {\rm ann}_R(y)$, then  $x-z-yz-y$ is a path between $x$ and $y$,  for every $z\in \mathfrak{p}_1\setminus \mathfrak{p}_2\cup {\rm ann}_R(y)$.

\textbf{Case 2.} $\mathrm {Min}(R)=\{\mathfrak{p}_1,\mathfrak{p}_2\}$. If $R$ is a reduced ring, then $\mathfrak{p}_1\cap \mathfrak{p}_2=(0)$ and by \cite[Corollary 2.4]{Huckaba}, $Z(R)=\mathfrak{p}_1\cup \mathfrak{p}_2$. It is not hard to see that, in this case,  $\Gamma^{^{\prime}}(R)=K_{{|\mathfrak{p}_1}^*|}\cup K_{{|\mathfrak{p}_2}^*|}$ and thus  $\Gamma^{^{\prime}}(R)$ is not connected. Hence suppose that $R$ is a non-reduced ring and consider the following subcases:

\textbf{Subcase 1.} $\mathfrak{p}_1,\mathfrak{p}_2\subseteq {\rm ann}_R(x)$. If $\mathfrak{p}_1 \subseteq {\rm ann}_R(y)$ or $\mathfrak{p}_2\subseteq {\rm ann}_R(y)$, then since ${\rm ann}_R(x)\nsubseteq {\rm ann}_R(y)$,  $x-z-y$ is  a path from $x$ to $y$, for  every  $z\in {\rm ann}_R(x)\setminus {\rm ann}_R(y)$.
If  $\mathfrak{p}_1 \nsubseteq {\rm ann}_R(y)$ and $\mathfrak{p}_2\nsubseteq {\rm ann}_R(y)$, then $x-z-yz-y$ is a path connecting $x$ and $y$, for every $z\in \mathfrak{p}_1\setminus \mathfrak{p}_2\cup {\rm ann}_R(y)$. If $\mathfrak{p}_1, \mathfrak{p}_2\subseteq {\rm ann}_R(y)$, the proof is similar.

\textbf{Subcase 2.} $\mathfrak{p}_1 \subseteq {\rm ann}_R(x)$ and $\mathfrak{p}_2\subseteq {\rm ann}_R(y)$. If $Z(R)\neq \mathfrak{p}_1\cup \mathfrak{p}_2$,
then  $x-z-y$ is  a path connecting $x$ and  $y$, for every $z\in Z(R)\setminus \mathfrak{p}_1\cup \mathfrak{p}_2$. If $Z(R)=\mathfrak{p}_1\cup \mathfrak{p}_2$, then  $\mathfrak{p}_1={\rm ann}_R(x)$ and ${\rm ann}_R(y)= \mathfrak{p}_2$.
Let $x+y=c$. We show that ${\rm ann}_R(x)\cap {\rm ann}_R(y)={\rm ann}_R(c)$. Clearly,  ${\rm ann}_R(x)\cap {\rm ann}_R(y)\subseteq {\rm ann}_R(c)$. If $r \in {\rm ann}_R(c) $ such that either $r\not\in {\rm ann}_R(x)$ or $r\not\in {\rm ann}_R(y)$, then $rx=-ry\neq 0$. This implies that ${\rm ann}_R(x)= {\rm ann}_R(y)$, a contradiction.
 Hence ${\rm ann}_R(x)\cap {\rm ann}_R(y)={\rm ann}_R(c)$ and so  $x-c-y$ is a path from $x$ to $y$. If $\mathfrak{p}_1 \subseteq {\rm ann}_R(y)$ and $\mathfrak{p}_2\subseteq {\rm ann}_R(x)$, the proof is similar.

\textbf{Case 3.} $\mathrm {Min}(R)=\{\mathfrak{p}\}$. If $x,y\not\in \mathfrak{p}$, then  ${\rm ann}_R(x)\cup {\rm ann}_R(y)= Z(R)$ implies that $Z(R)\subseteq \mathfrak{p}$, a contradiction. Thus let $x\in \mathfrak{p}$. Hence   $x+y=c\in Z(R)$ and so  $x-c-y$ is a path from $x$ to $y$.

The above argument shows that $\Gamma^{^{\prime}}(R)$ is not connected if and only if $R$ is reduced with $|\mathrm {Min}(R)|=2$ and $\mathrm{diam}(\Gamma^{^{\prime}}(R))\leq 4$, if $\Gamma^{^{\prime}}(R)$ is connected.
 }
\end{proof}

The next result states that if
 $\mathrm{diam}(\Gamma^{^{\prime}}(R))\leq 2$, then  $Z({\rm ann}_R(x))=Z(R)$, for some $x\in Z(R)$.

\begin{thm}\label{starccAGgg}
  Let $R$ be a ring. If
 $\mathrm{diam}(\Gamma^{^{\prime}}(R))\leq 2$, then $Z({\rm ann}_R(x))=Z(R)$, for some $x\in Z(R)$.
\end{thm}
\begin{proof}
{ Assume that  $\mathrm{diam}(\Gamma^{^{\prime}}(R))\leq 2$. If there exists $x\in Z(R)$ such that $x$ is adjacent to every other vertex, then
${\rm ann}_R(x)\cap {\rm ann}_R(y)\neq 0$, for every $y\in Z(R)$. We show that $Z({\rm ann}_R(x))=Z(R)$. Let $y\in Z(R)\setminus Z({\rm ann}_R(x))$. Since ${\rm ann}_R(x)\cap {\rm ann}_R(y) \neq 0$, $ay=0$, for some $0\neq a \in {\rm ann}_R(x) $. This implies that $y\in Z({\rm ann}_R(x))$, a contradiction. Thus $Z({\rm ann}_R(x))=Z(R)$.
Hence we may suppose that $\mathrm{diam}(\Gamma^{^{\prime}}(R))=2$ and ${\rm ann}_R(a)\cap {\rm ann}_R(b)= 0$, for some $a,b \in Z(R)^*$. Since $\mathrm{diam}(\Gamma^{^{\prime}}(R))=2$, there exists $c\in Z(R)^*$ such that $a-c-b$ is a path from $a$ to $b$ and so ${\rm ann}_R(a)\cup {\rm ann}_R(b)\subseteq {\rm ann}_R(c)$. If ${\rm ann}_R(c)\cap {\rm ann}_R(y) \neq 0$, for every $y\in Z(R)^*$, then $Z({\rm ann}_R(c))=Z(R)$ and if ${\rm ann}_R(c)\cap {\rm ann}_R(y) = 0$, for some $y\in Z(R)^*$, then there exists $d\in Z(R)^*$ such that $c-d-y$ is a path from $c$ to $y$. If ${\rm ann}_R(d)\cap {\rm ann}_R(z) \neq 0$, for every $z\in Z(R)^*$, then  $Z({\rm ann}_R(d))=Z(R)$ and if ${\rm ann}_R(d)\cap {\rm ann}_R(z) = 0$,
 for some $z\in Z(R)^*$, then there exists $e\in Z(R)^*$ such that $d-e-z$ is a path from $e$ to $z$. By continuing this procedure,  $Z({\rm ann}_R(x))=Z(R)$, for some $x\in Z(R)$.}
\end{proof}

The next example shows that the converse of Theorem \ref{starccAGgg} does not hold.
\begin{example}

Let $R=\mathbb{Z}_4\times  \mathbb{Z}_4$ and $x=(2,2)$. Then $Z({\rm ann}_R(x))=Z(R)$, but $\mathrm{diam}(\Gamma^{^{\prime}}(R))=4$.
\end{example}
\begin{thm}\label{starccAGd}
  Let $R$ be a ring. Then
 $\mathrm{diam}(\Gamma^{^{\prime}}(R))\leq 2$ if and only if for every  two distinct elements $x,y\in Z(R)^*$ one of  the following statements holds.

$(1)$ for some $z\in Z(R)$, $  {\rm ann}_R(z)\subseteq {\rm ann}_R(x)\cap {\rm ann}_R(y) $, or

$(2)$ $  {\rm ann}_R( {\rm ann}_R(x)+ {\rm ann}_R(y))\neq 0 $.
\end{thm}
\begin{proof}
{Assume that  $\mathrm{diam}(\Gamma^{^{\prime}}(R))\leq 2$, $x,y\in Z(R)^*$ and there is no  $z\in Z(R)$ with the property  ${\rm ann}_R(z)\subseteq {\rm ann}_R(x)\cap {\rm ann}_R(y) $. Since  $d(x,y)=2$,  there exists   $c\in Z(R)$ such that  ${\rm ann}_R(x)\cup {\rm ann}_R(y) \subseteq {\rm ann}_R(c)$. This implies that ${\rm ann}_R( {\rm ann}_R(x)+ {\rm ann}_R(y))\neq 0 $.

The converse is clear.}
\end{proof}
Next, we study the girth of $(\Gamma^{^{\prime}}(R))$.
 \begin{thm}\label{girth}
  Let $R$ be a ring. Then   gr$(\Gamma^{^{\prime}}(R))=\{3,6,\infty \} $.
\end{thm}
\begin{proof}
{  Let $R$ be a ring and consider the following cases.

\textbf{Case 1.} $R$ is non-reduced. If $R\cong R_1\times R_2$, where $R_1$ and $R_2$ are two rings, then without loss of generality, one may suppose that
$a \in R_1$ is a non-zero element of $\mathrm{Nil}(R_1)$. Therefore, there exist $u_1, u_2 \in U(R_1)$. Thus the vertices $(u_1,0),(u_2,0),(a,0)$ forms a triangle, i.e., gr$(\Gamma^{^{\prime}}(R))=3$.  Hence we may assume that $R$ is indecomposable. If $Z(R)\neq \mathrm{Nil}(R)$, then there exists an element $x\in Z(R)\setminus \mathrm{Nil}(R)$. Thus $x-x^2-x^3-x$ is a triangle and so gr$(\Gamma^{^{\prime}}(R))=3$. Now, we assume that $Z(R)=\mathrm{Nil}(R)$. If $\mathrm{Nil}(R)^3 \neq (0)$, then there exist $x,y,z\in \mathrm{Nil}(R)$ such that $xyz\neq 0$ and so $x-xy-xyz-x$ is a triangle, i.e., $(\Gamma^{^{\prime}}(R))=3$. If $\mathrm{Nil}(R)^3= (0)$, then ${\rm ann}_R(x)=\mathrm{Nil}(R)$, for some $x\in \mathrm{Nil}(R)^*$. Hence $x$ is adjacent to every other vertex and so gr$(\Gamma^{^{\prime}}(R))=3$ or $\infty$.

\textbf{Case 2.} $R$ is reduced. If $R$  is indecomposable, then for every $x\in Z(R)^*$, we have $x^n\neq x^m$, where $n,m$ are distinct positive integers. Thus  $x-x^2-x^3-x$ is a triangle and so gr$(\Gamma^{^{\prime}}(R))=3$. Now assume that $R\cong R_1\times \cdots\times R_n$, where every $R_i$, $1\leq i \leq n$, is a ring and $n$ is a positive integer. If $n\geq 4$, then the vertices $(0,1,1,\dots,1), (0,0,1,\dots,1),(0,0,0,1,\dots,1)$ forms a triangle and so gr$(\Gamma^{^{\prime}}(R))=3$. To complete the proof consider the following subcases.

\textbf{Subcase 1.} $R\cong R_1\times R_2\times R_3$ and every $R_i$, $1\leq i \leq 3$,  is  indecomposable.
Suppose that $R_1$ is not a field and $a\in R$ is non-unit and non-zero. Since $R_1$ is indecomposable, the vertices $(a,0,0), (a^2,0,0),(a^3,0,0)$ forms a triangle, i.e.,  gr$(\Gamma^{^{\prime}}(R))=3$. Hence suppose that every $R_i$,  $1\leq i\leq 3$, is a field. If  $|R_3|\geq 3$, then  the vertices $(0,0,1),(0,0,x),(0,1,1)$ forms a triangle, where $x\in R_3^*$ and so gr$(\Gamma^{^{\prime}}(R))=3$. Thus one may assume that $R\cong \mathbb{Z}_2\times\mathbb{Z}_2\times\mathbb{Z}_2$. In this case, $\Gamma^{^{\prime}}(R)=C_6$ and thus $gr(\Gamma^{^{\prime}}(R))=6$.

\textbf{Subcase 2.} $R\cong R_1\times R_2$ and every $R_i$, $1\leq i \leq 2$,  is  indecomposable.
 By a similar argument to that in Subcase 1, if for some $1\leq i\leq 2$, $R_i$ is not a field, then  gr$(\Gamma^{^{\prime}}(R))=3$. Hence we may assume that  $R_1$  and $R_2$ are two fields. In this case, it is not hard to see that $\Gamma^{^{\prime}}(R)=K_{|R_1^*|}\cup K_{|R_2^*|}$. If $R\cong \mathbb{Z}_3\times\mathbb{Z}_3$ or $R\cong \mathbb{Z}_2\times\mathbb{Z}_3$ or $R\cong \mathbb{Z}_2\times\mathbb{Z}_2$, then gr$(\Gamma^{^{\prime}}(R))=\infty$ and otherwise gr$(\Gamma^{^{\prime}}(R))=3$.
 }
\end{proof}

In light of Theorem \ref{girth}, the last result of this section classifies all rings $R$ in terms of gr$(\Gamma^{^{\prime}}(R))$.
 \begin{cor}\label{cgirth}
 Let $R$ be a ring. Then the following statements hold.

$(1)$   gr$(\Gamma^{^{\prime}}(R))=6$ if and only if $R\cong \mathbb{Z}_2\times\mathbb{Z}_2\times\mathbb{Z}_2$.

$(2)$  gr$(\Gamma^{^{\prime}}(R))=\infty$ if and only if either $|Z(R)|\leq 3$ or $R$ is ring isomorphic to one of the following rings. $$\mathbb{Z}_3\times\mathbb{Z}_3, \mathbb{Z}_2\times\mathbb{Z}_3, \mathbb{Z}_2\times\mathbb{Z}_2$$

$(3)$  gr$(\Gamma^{^{\prime}}(R))=3$ if and only if $|Z(R)|\geq 4$ and $R$ is not  ring isomorphic to any of the following rings.
 $$\mathbb{Z}_3\times\mathbb{Z}_3, \mathbb{Z}_2\times\mathbb{Z}_3, \mathbb{Z}_2\times\mathbb{Z}_2, \mathbb{Z}_2\times\mathbb{Z}_2\times\mathbb{Z}_2$$
\end{cor}
\begin{proof}
{$(1)$ and $(3)$ are obtained, by proof of Theorem \ref{girth}.

$(2)$ One side is clear. To prove other side, assume that gr$(\Gamma^{^{\prime}}(R))=\infty$. By proof of Theorem \ref{girth}, if $R$ is non-reduced, $Z(R)=\mathrm{Nil}(R)$ and $\mathrm{Nil}(R)^3= (0)$. We show that $\mathrm{Nil}(R)^2= (0)$.  Since $\mathrm{Nil}(R)^3= (0)$ for some $x\in \mathrm{Nil}(R)^*$, ${\rm ann}_R(x)=\mathrm{Nil}(R)$ and thus $x$ is adjacent to every other vertex. Now, let $y\in \mathrm{Nil}(R)\setminus \{x,0\}$. If $yz\neq 0$ for some $z\in \mathrm{Nil}(R)$, then $x-yz-z-x$ is a triangle, a contradiction. So $\mathrm{Nil}(R)^2= (0)$. This means that $\Gamma^{^{\prime}}(R)$ is a complete graph and since gr$(\Gamma^{^{\prime}}(R))=\infty$, we deduce that $|Z(R)|\leq 3$. If $R$ is reduced, the result follows from Subcase 2 of Theorem \ref{girth}.}
\end{proof}
%%%%%%%%%%%%%%
%%%%%%%%%%%%%%%%%%%%%%%%%%%%%%%%%%%%%%%%%%%%%%%%%%%%%%%%%%%%%%%%%%%%%%%%%%%%%%%%%%%%%%%%%%%%%%%%%%%%%%%%%%%%%%%%%%%%%%%%%%%%%%%%%%%%%%%%%%%%%%%%%%%%%%

%%%%%%%%%%%%%%%%%%%%%%%%%%%%%%%%%%%%%%%%%%%%%%%%%%%%%%%%%%%%%%%%%%%%%%%%%%%%%%%%%%%%%%%%%%%%%%%%%%%%%%%%%%%%%%%%%%%%%%%%%%%%%%%%%%%%%%%%%%%%%%%%%%%%%%
{\section{ When $\Gamma^{^{\prime}}(R)$  and $\overline{AG(R)}$ are identical?}}\vspace{-2mm}

This section is devoted
to characterize rings whose inclusion graphs of annihilators are identical to
the complement of annihilator graphs. We start with the following simple example which shows that  $\Gamma^{^{\prime}}(R)\neq \overline{AG(R)}$ and $\Gamma^{^{\prime}}(R) =\overline{AG(R)}$ may occur.

\begin{example}
{\rm $(1)$ Let $R=\mathbb{Z}_2\times  \mathbb{Z}_4$. Then $\Gamma^{^{\prime}}(R)\neq \overline{AG(R)}$.

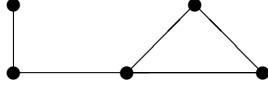
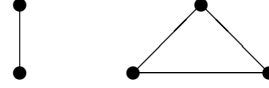
\unitlength=1.5mm
\begin{picture}(60,40)(-20,-20)
\put (-16,0){\circle*{1.2}}
\put (-16,0){\line (1,0){12}}
\put (-16,0){\line (0,1){6}}
\put (0,6){\circle*{1.2}}
\put (-16,6){\circle*{1.2}}
\put (-6,0){\circle*{1.2}}
\put (-6,0){\line (1,1){6}}
\put (-6,0){\line (1,0){12}}
\put (-6,0){\line (1,0){12}}
\put (6,0){\circle*{1.2}}
\put (6,0){\line (-1,1){6}}
\put (44,0){\circle*{1.2}}
\put (44,0){\line (0,1){6}}
\put (60,6){\circle*{1.2}}
\put (44,6){\circle*{1.2}}
\put (54,0){\circle*{1.2}}
\put (54,0){\line (1,1){6}}
\put (54,0){\line (1,0){12}}
\put (54,0){\line (1,0){12}}
\put (66,0){\circle*{1.2}}
\put (66,0){\line (-1,1){6}}
\put (-12,-15){$\Gamma^{^{\prime}}(\mathbb{Z}_2\times  \mathbb{Z}_4)$}
\put (48,-15){$  \overline{AG(\mathbb{Z}_2\times  \mathbb{Z}_4)} $}

\end{picture}

$(2)$ Let $R=\mathbb{Z}_2\times  \mathbb{Z}_2\times \mathbb{Z}_2$. Then $\Gamma^{^{\prime}}(R)= \overline{AG(R)}$.

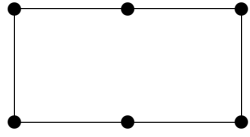
\begin{picture}(60,40)(-20,-20)
\put (20,10){\circle*{1.2}}
\put (20,0){\circle*{1.2}}
\put (30,10){\circle*{1.2}}
\put (30,0){\circle*{1.2}}
\put (10,10){\circle*{1.2}}
\put (10,0){\circle*{1.2}}
\put (10,0){\line (1,0){20}}
\put (10,10){\line (1,0){20}}
\put (10,0){\line (0,1){10}}
\put (30,0){\line (0,1){10}}

\put (12,-12){$ \Gamma^{^{\prime}}(R)= \overline{AG(R)} $}
\end{picture}
 }
\end{example}
The following lemma will be used frequently in this section.
\begin{lem}\label{identical Lemma}
  Let $R$ be a ring. Then the following statements hold.

  $(1)$  $x-y$ is an edge of $AG(R)$ if and only if $Rx\cap {\rm ann}_R(y)\neq (0)$ and $Ry\cap {\rm ann}_R(x)\neq (0)$.

  $(2)$ If $x-y$ is an edge of $\overline{AG(R)}$, then  ${\rm ann}_R(x)\subseteq {\rm ann}_R(y)$ or ${\rm ann}_R(y)\subseteq {\rm ann}_R(x)$.

  $(3)$ If $x-y$ is an edge of $\overline{AG(R)}$, then $x-y$ is an edge of $\Gamma^{^{\prime}}(R)$.

   $(4)$ $AG(R)[\mathrm{Nil}(R)^*]$ is a complete subgraph of $AG(R)$.
\end{lem}
\begin{proof}
{$(1)$ is obtained, by \cite[Lemma 2.1]{Coloring}.

$(2)$ is obtained, by part $(3)$ of \cite[Lemma 2.1]{Badawi}.

$(3)$ is obtained, by part $(2)$.

$(4)$ is obtained, by part $(3)$ of \cite[Theorem 3.10]{Badawi}.
}
\end{proof}
The next result states $\Gamma^{^{\prime}}(R)= \overline{AG(R)}$, if $R$ is reduced.
\begin{thm}\label{reduced identical}
  Let $R$ be a reduced ring. Then $\Gamma^{^{\prime}}(R)= \overline{AG(R)}$.
\end{thm}
\begin{proof}
{By part $(3)$ of Lemma \ref{identical Lemma}, $\overline{AG(R)}$ is a subgraph of $\Gamma^{^{\prime}}(R)$ and so we need only to show that  $\Gamma^{^{\prime}}(R)$ is a subgraph of $\overline{AG(R)}$. Let $x-y$ be an edge of $\Gamma^{^{\prime}}(R)$. Without loss of generality, ${\rm ann}_R(x)\subseteq {\rm ann}_R(y)$. Since $R$ is reduced, $Rx\cap {\rm ann}_R(x)= (0)$ and thus $Rx\cap {\rm ann}_R(y)= (0)$. By part $(1)$ of Lemma \ref{identical Lemma}, $x-y$ is not an edge of $AG(R)$, i.e., $x-y$ is an edge of $\overline{AG(R)}$.
}
\end{proof}
Next, we study non-reduced rings $R$ with $\Gamma^{^{\prime}}(R)= \overline{AG(R)}$.
\begin{thm}\label{non reduced identical1}
  Let $R$ be a non-reduced ring. Then  $\Gamma^{^{\prime}}(R)= \overline{AG(R)}$ if and only if all of the following statements hold.

 $(1)$ $\mathrm{Nil}(R)=\{0,a\}$, for some $a \in Z(R)$.

 $(2)$ $Rx+ {\rm ann}_R(x)\nsubseteq {\rm ann}_R(a)$,  for every $x \in Z(R)\setminus \{0,a\}$.

\end{thm}
\begin{proof}
{First, suppose that $\Gamma^{^{\prime}}(R)= \overline{AG(R)}$.

Let  $a\in \mathrm{Nil}(R)^*$. If  $b\in Ra\setminus \{0,a\}$, then $Rb\subseteq Ra$ and thus $a-b$ is an edge of $\Gamma^{^{\prime}}(R)$ which is not an edge of $\overline{AG(R)}$, by part $(4)$ of Lemma  \ref{identical Lemma}. This implies that $ Ra= \{0,a\}$, for every $a\in \mathrm{Nil}(R)^*$ and so ${\rm ann}_R(a)$ is a maximal ideal of $R$. Let $a ,b\in \mathrm{Nil}(R)^*$ and $a\neq b$. We know that ${\rm ann}_R(a)$ and ${\rm ann}_R(b)$ are maximal ideals. If ${\rm ann}_R(a)={\rm ann}_R(b)$, then $a-b$ is an edge of $\Gamma^{^{\prime}}(R)$ which is not an edge of $\overline{AG(R)}$, by part $(4)$ of Lemma \ref{identical Lemma}, a contradiction. Hence assume that  ${\rm ann}_R(a)=\mathfrak{m}_1$, ${\rm ann}_R(b)=\mathfrak{m}_2$ and $\mathfrak{m}_1\neq\mathfrak{m}_2$. By the above argument, it is not hard to see that  ${\rm ann}_R(a+b)=\mathfrak{m}_3$ is a maximal ideal such that $\mathfrak{m}_3\neq\mathfrak{m}_1$ and $\mathfrak{m}_3\neq\mathfrak{m}_2$. If $r\in {\rm ann}_R(a) \cap {\rm ann}_R(b)=\mathfrak{m}_1\cap\mathfrak{m}_2 $, then $r(a+b)=0$, and so $ {\rm ann}_R(a) \cap {\rm ann}_R(b)\subseteq {\rm ann}_R(a+b) $, a contradiction. This completes the proof of statement $(1)$.

If $Rx+ {\rm ann}_R(x)\subseteq {\rm ann}_R(a)$, for some $x \in Z(R)\setminus \{0,a\}$, then $a-x$ is an edge of $\Gamma^{^{\prime}}(R)$ which is not an edge of $\overline{AG(R)}$, a contradiction. Hence $Rx+ {\rm ann}_R(x)\subseteq {\rm ann}_R(a)$,   for every $x \in Z(R)\setminus \{0,a\}$.

Conversely, suppose that the statements $(1)$ and $(2)$ hold.  We show that $\Gamma^{^{\prime}}(R)= \overline{AG(R)}$.

Let $x-y$ be an edge of $\Gamma^{^{\prime}}(R)$ and ${\rm ann}_R(x)\subseteq {\rm ann}_R(y)$. By part $(1)$ of Lemma  \ref{identical Lemma}, we need only show that  $Ry\cap {\rm ann}_R(x)= (0)$. If $Ry\cap {\rm ann}_R(x)\neq (0)$, then there exists an element $r\in R$ such that $ryx=0$ and $ry\neq 0$. Thus $ry\in {\rm ann}_R(x)$. On the other hand,  since ${\rm ann}_R(x)\subseteq {\rm ann}_R(y)$, we deduce $ryy=ryry=0$. This means that $ry=a$ and so $x\in {\rm ann}_R(a)$. Since $ry=a$, we conclude that ${\rm ann}_R(y)\subseteq {\rm ann}_R(a)$, and  ${\rm ann}_R(x)\subseteq {\rm ann}_R(y)$ implies that $Rx+ {\rm ann}_R(x)\subseteq {\rm ann}_R(a)$. This contradicts the statement $(2)$. Therefore, $Ry\cap {\rm ann}_R(x)= (0)$, as desired.
}
\end{proof}
\begin{cor}\label{identical  cor}
  Let $R$ be a non-reduced ring.  If $\Gamma^{^{\prime}}(R)= \overline{AG(R)}$, then the following statements hold.

  $(1)$ $R$ is  indecomposable.

  $(2)$ If $R$ is an Artinian ring, then  $R=\mathbb{Z}_4$ or $R=\mathbb{Z}_2[X]/(X^2)$.

  $(3)$ $ {\rm ann}_R(x^2)={\rm ann}_R(x^3)$, for every $x \in Z(R)$.

  $(4)$ $|\mathrm {Max}(R)\cap\mathrm {Ass}(R)|=1$.

  $(5)$  If $R$ is a local ring, then  $R=\mathbb{Z}_4$ or $R=\mathbb{Z}_2[X]/(X^2)$.

  $(6)$ If there exists  $x \in Z(R)^*$ such that $ {\rm ann}_R(x)$ is an essential ideal, then $x$ is the only nilpotent element of $R$.
\end{cor}
\begin{proof}
{$(1)$ Suppose that $R\cong R_1\times R_2$, where $R_1$ and $R_2$ are two rings. Without loss of generality, assume that $|\mathrm{Nil}(R_1)|\geq 2$. Let $a\in \mathrm{Nil}(R_1)^*$, $x=(a,0)$ and $y=(a,1)$. Then $x-y$ is an edge of $\Gamma^{^{\prime}}(R)$ which is not an edge of $\overline{AG(R)}$, a contradiction. Thus  $R$ is indecomposable.

$(2)$  By part $(1)$,  $R$ is indecomposable. Since $R$ is an Artinian ring, we deduce that $R$ is local. By part $(4)$ of Lemma \ref{identical Lemma},  $\overline{AG(R)}$ is null. Also, by Theorem \ref{diam}, $\Gamma^{^{\prime}}(R)$ is connected. This implies that $\Gamma^{^{\prime}}(R)$ has only one vertex and so $|Z(R)|=2$. Thus $R=\mathbb{Z}_4$ or $R=\mathbb{Z}_2[X]/(X^2)$.

$(3)$ Suppose that  there exists $x \in Z(R)$ such that $ {\rm ann}_R(x^2)\neq {\rm ann}_R(x^3)$. By part $(1)$, $R$ is indecomposable and so $x\neq x^2$. This implies that $x-x^2$ is an edge of $\Gamma^{^{\prime}}(R)$. Since $ {\rm ann}_R(x^2)\neq {\rm ann}_R(x^3)$,  there exists an element $r\in R$ such that $rx^3=0$ and $rx^2\neq 0$. This means that  $Rx^2\cap {\rm ann}_R(x)\neq (0)$. Also,  $rx x^2=0$ implies that $Rx\cap {\rm ann}_R(x^2)\neq (0)$. Thus by part $(1)$ of Lemma \ref{identical Lemma}, $x-x^2$ is an edge of ${AG(R)}$, a contradiction. Hence $ {\rm ann}_R(x^2)={\rm ann}_R(x^3)$,  for every $x \in Z(R)$.

$(4)$ is easily obtained, by proof of Theorem \ref{non reduced identical1}.

$(5)$ Let  $\mathfrak{m}$ be the only maximal ideal of $R$. By part $(4)$,   $\mathfrak{m}={\rm ann}_R(x)$, for some $x\in Z(R)^*$. This, together with the fact that $Z(R)=\mathfrak{m}$, implies that $x$ is adjacent to all other vertices in ${AG(R)}$. Thus $\overline{AG(R)}$ is not connected, if $|Z(R)^*|\geq 2$. This contradicts Theorem \ref{diam} and so $|Z(R)^*|=1$, i.e., $R=\mathbb{Z}_4$ or $R=\mathbb{Z}_2[X]/(X^2)$.

$(6)$ Let $x \in Z(R)^*$ and $ {\rm ann}_R(x)$ be an essential ideal of $R$. If $x^2\neq 0$, then  part $(1)$ implies that $x\neq x^2$, as $R$ is indecomposable. This implies that $x-x^2$ is an edge of $\Gamma^{^{\prime}}(R)$. Since  $ {\rm ann}_R(x)$ is essential, $ {\rm ann}_R(x^2)$ is an essential ideal of $R$, too. Thus $Rx\cap {\rm ann}_R(x^2)\neq (0)$ and $Rx^2\cap {\rm ann}_R(x)\neq (0)$, a contradiction, by part $1$ of Lemma \ref{identical Lemma}. The result now follows from part $1$ of Theorem \ref{non reduced identical1}.}
\end{proof}

{\section{Complete inclusion graphs of annihilators}}\vspace{-2mm}
In this section, we study complete inclusion graphs of annihilators.
We start with the following result.

\begin{thm}\label{complet reduc}
  Let $R$ be a reduced ring that is not an integral domain.  Then $\Gamma^{^{\prime}}(R)$ is not a complete graph.
 \end{thm}
\begin{proof}
{Since $R$ is not an integral domain, $|\mathrm {Min}(R)|\geq 2$. Suppose that $\mathfrak{p}_1,\mathfrak{p}_2 \in \mathrm {Min}(R)$, $x\in \mathfrak{p}_1\setminus\mathfrak{p}_2$ and $y\in \mathfrak{p}_2\setminus\mathfrak{p}_1$. Clearly, ${\rm ann}_R(x)\subseteq\mathfrak{p}_2 $ and ${\rm ann}_R(y)\subseteq\mathfrak{p}_1 $. By \cite[Corollary 2.2]{Huckaba}, there exist elements $a\in {\rm ann}_R(x)$ and $b\in {\rm ann}_R(y)$ such that $a\not\in \mathfrak{p}_1$, $b\not\in \mathfrak{p}_2$. This implies that ${\rm ann}_R(x)\nsubseteq {\rm ann}_R(y)$ and ${\rm ann}_R(y)\nsubseteq {\rm ann}_R(x)$. Thus $x$ is not adjacent to $y$ and $\Gamma^{^{\prime}}(R)$ is not complete.
}
\end{proof}

In the continuing by $K_\infty$ we mean a complete graph with infinitely many verices; also,   $K_{1,\infty}$ denotes  a star graph with infinite vertices. 
\begin{thm}\label{starccAG}
  Let $R$ be a non-reduced ring and $AG(R)$ is a star graph.
  Then the following statements are equivalent.

$(1)$ $AG(R) \neq \Gamma^{^{\prime}}(R)$.

$(2)$ $\Gamma^{^{\prime}}(R)=K_\infty$.

$(3)$  $AG(R)=K_{1,\infty}$.
\end{thm}
\begin{proof}
{ $(1)\Rightarrow(2)$
Since $AG(R) \neq \Gamma^{^{\prime}}(R)$, $|Z(R)^*|>2$ and thus, by  \cite[Theorem 3.18]{Badawi}, $\mathrm{Nil}(R)=\{0,w\}$ is a prime ideal of $R$, for some $w\in Z(R)^*$. This implies that ${\rm ann}_R(w)=Z(R) $ and  ${\rm ann}_R(x)=\{0,w\}$,  for every $x\in Z(R)\setminus \{0,w\}$. It is not hard to see that $\Gamma^{^{\prime}}(R)=K_\infty$.

$(2)\Rightarrow(3)$ Since $\Gamma^{^{\prime}}(R)=K_\infty$, we deduce that $|Z(R)|=\infty$. Since  $AG(R)$ is a star graph,  \cite[Theorem 3.18]{Badawi} implies that  $AG(R)=K_{1,\infty}$.

$(3)\Rightarrow(1)$ Since $AG(R)$ is a star graph,   ${\rm ann}_R(x)=\mathrm{Nil}(R)$ or ${\rm ann}_R(x)=Z(R)$, for every $x\in Z(R)^*$. Thus $\Gamma^{^{\prime}}(R)=K_\infty$ an so $AG(R) \neq \Gamma^{^{\prime}}(R)$.
 }
\end{proof}
%Theorem \ref{starccAG} follows the following corollary which completely  recognizes star inclusion graphs of annihilators.
%\begin{cor}\label{zcgirth}
% Let $R$ be a ring. Then the following statements are equivalent.
%
%$(1)$ $\Gamma^{^{\prime}}(R)$ is a star graph.
%
%$(2)$  gr$(\Gamma^{^{\prime}}(R))=\infty$  and $R$ is  a non-reduced ring.
%
%$(3)$  $\Gamma^{^{\prime}}(R)= K_{1,1}$.
%\end{cor}
%\begin{proof}
%{$(1)\Rightarrow(2)$ Since $\Gamma^{^{\prime}}(R)$ is a star graph, girth$(\Gamma^{^{\prime}}(R))=\infty$ and $R$ is  a non-reduced ring, by Corollary \ref{cgirth}.
%
%$(2)\Rightarrow(3)$ is clear, by Corollary \ref{cgirth}.
%
%$(3)\Rightarrow(1)$ is clear.??????????????che ertebati be theorem ghabl darad???? aya az corollary girth vazeh nist????
%}
%\end{proof}

\begin{thm}\label{identical  cor}
  Let $R$ be a non-reduced ring and  $\Gamma^{^{\prime}}(R)$ is a complete graph. Then the following statements hold.

  $(1)$ $R$ is indecomposable.

  $(2)$ $\mathrm {Ass}(R)\neq\emptyset$.
\end{thm}
\begin{proof}
{$(1)$ Suppose that $R\cong R_1\times R_2$, where $R_1$ and $R_2$ are two rings. Without loss of generality, assume that $|\mathrm{Nil}(R_1)|\geq 2$. Let  $a\in \mathrm{Nil}(R_1)^*$, $u\in U(R_1)$, $x=(u,0)$ and $y=(a,1)$. Then  $x-y$ is not an edge of $\Gamma^{^{\prime}}(R)$, a contradiction. Thus  $R$ is  indecomposable.

$(2)$ Since  $\Gamma^{^{\prime}}(R)$ is a complete graph,   ${\rm ann}_R(x)\subseteq {\rm ann}_R(y)$ or ${\rm ann}_R(y)\subseteq {\rm ann}_R(x)$, for every two distinct elements $x, y\in Z(R)^*$. This, together with the fact $Z(R)=\cup_{x\in Z(R)^*}{\rm ann}_R(x)$, implies that $Z(R)={\rm ann}_R(a)$, for some $a\in Z(R)^*$. This means that $Z(R)\in \mathrm {Ass}(R)$ and so $\mathrm {Ass}(R)\neq\emptyset$.}
\end{proof}

Let $R$ be a ring and $I_1\subseteq I_2\subseteq\cdots $ be a chain of ideals of $R$. Then we say that the set $\{I_1,I_2,\dots \}$ is a chain set of ideals. Also, let $\sum=\{{\rm ann}_R(x) \,\,|\,\,x\in Z(R)^*\}$.
\begin{thm}\label{complet chain }
  Let $R$ be a ring. Then the following statements are equivalent.

$(1)$ $\Gamma^{^{\prime}}(R)$ is a complete graph.

$(2)$ $\sum$ is a chain  of  ideals of $R$.
 \end{thm}
\begin{proof}
{$(1)\Rightarrow (2)$ Since $\Gamma^{^{\prime}}(R)$ is a complete graph,  ${\rm ann}_R(x)\subseteq {\rm ann}_R(y)$ or ${\rm ann}_R(y)\subseteq {\rm ann}_R(x)$, for every two distinct ${\rm ann}_R(x), {\rm ann}_R(y)\in \sum$. This implies that $\sum$ is a chain set of ideals.

$(2)\Rightarrow (1)$ is clear.
}
\end{proof}
To state the last result of this paper, the following lemma is needed.

\begin{lem}\label{subgraph3}
 Let $R$ be a ring. Then the following statements hold.

$(1)$  ${\rm ann}_R(x)$ is an essential ideal of $R$, for every $x\in \mathrm{Nil}(R)$.

$(2)$ If there exist two distinct elements  $x,y\in Z(R)^*$ such that ${\rm ann}_R(x)$ and ${\rm ann}_R(y)$ are essential ideals of $R$, then $x$ is not adjacent to $y$ in
$\overline{AG(R)}$.
\end{lem}
\begin{proof}
{$(1)$ Suppose to the contrary that, there exists an ideal $I$ such that $I\cap {\rm ann}_R(x)=(0)$. This implies that $ax\neq 0$, for every $a\in I^*$. Obviously, $ax\in I$ and so $axx=ax^2\neq 0$. By continuing this procedure,  $ax^n\neq 0$, for every positive integer $n$, a contradiction. Hence ${\rm ann}_R(x)$ is an essential ideal of $R$, as desired.

 $(2)$ Suppose that $x,y\in Z(R)^*$ and  ${\rm ann}_R(x)$, ${\rm ann}_R(y)$ are essential ideals of $R$. It is enough to show that $x$ is adjacent to $y$ in
$AG(R)$.  Since  ${\rm ann}_R(x)$ and ${\rm ann}_R(y)$ are essential ideals of $R$,  $Rx\cap {\rm ann}_R(y)\neq (0)$ and  $Ry\cap {\rm ann}_R(x)\neq (0)$.  Thus by Lemma \cite[Lemma 2.1]{Coloring}, $x$ is adjacent to $y$ in $AG(R)$.}
\end{proof}

\begin{thm}\label{complet chain2}
  Let $R$ be a ring and $|\sum|= 2$. Then the following statements are equivalent:

$(1)$ ${\rm ann}_R(Z(R))$ is a prime ideal of $R$.

$(2)$  $\Gamma^{^{\prime}}(R)=K_{|Z(R)^*|}$ and $ \overline{AG(R)}=\overline{K}_{|{\rm ann}_R(Z(R))^*|}\cup K_{|Z(R)\setminus {\rm ann}_R(Z(R))|}$.

$(3)$ ${\rm ann}_R(Z(R))=\mathrm{Nil}(R)$ and  $|Z(R)|=\infty$.

$(4)$ $R$ is a non-reduced ring and $\sum=\mathrm {Ass}(R)$.
 \end{thm}
\begin{proof}
{$(1)\Rightarrow (2)$ Since ${\rm ann}_R(Z(R))$ is a prime ideal of $R$, either   ${\rm ann}_R(x)= {\rm ann}_R(Z(R))$ or ${\rm ann}_R(x)=Z(R)$,  for every $x\in Z(R)^*$. This implies that ${\rm ann}_R(x)\subseteq {\rm ann}_R(y)$ or ${\rm ann}_R(y)\subseteq {\rm ann}_R(x)$, for every two distinct $x, y\in Z(R)^*$. Thus $\Gamma^{^{\prime}}(R)$ is a complete graph. Also,it is not hard to see that
$ \overline{AG(R)}=\overline{K}_{|{\rm ann}_R(Z(R))^*|}\cup K_{|Z(R)\setminus {\rm ann}_R(Z(R))|}$.

$(2)\Rightarrow (3)$ If $|Z(R)|<\infty$, then $R$ is an Artinian ring. Since  $\Gamma^{^{\prime}}(R)$ is a complete graph, by Theorem \ref{identical  cor}, $R$ is indecomposable.  This, together with $R$ is an Artinian ring, implies that $R$ is an Artinian local ring. Hence by Lemma \ref{subgraph3}, $\overline{AG(R)}$ is  a null graph, a contradiction. So $|Z(R)|=\infty$.

$(3)\Rightarrow (4)$ Since $|\sum|= 2$,  we deduce that $Z(R)={\rm ann}_R(x)\cup {\rm ann}_R(y)$, for some $x,y\in Z(R)$. Since $R$ is a non-reduced ring, ${\rm ann}_R(x)\subseteq {\rm ann}_R(y)$ or ${\rm ann}_R(y)\subseteq {\rm ann}_R(x)$. Without loss  of generality, suppose that ${\rm ann}_R(y)\subseteq {\rm ann}_R(x)$. Thus $Z(R)={\rm ann}_R(x)$. It is enough to show that ${\rm ann}_R(y)$  is a prime ideal of $R$. We know that  ${\rm ann}_R(Z(R))={\rm ann}_R(y)$. We claim that ${\rm ann}_R(Z(R))=\mathrm{Nil}(R)$. Since ${\rm ann}_R(Z(R))\subseteq \mathrm{Nil}(R)$, we need only show that  $\mathrm{Nil}(R)\subseteq {\rm ann}_R(Z(R))$. Suppose that $a\in \mathrm{Nil}(R)\setminus {\rm ann}_R(Z(R))$. Since $|\sum|= 2$,
${\rm ann}_R(a)={\rm ann}_R(Z(R))$. Thus ${\rm ann}_R(Z(R))$ is an essential ideal of $R$. This contradicts the hypothesis $ \overline{AG(R)}=\overline{K}_{|{\rm ann}_R(Z(R))^*|}\cup K_{|Z(R)\setminus {\rm ann}_R(Z(R))|}$ and so the claim is proved. Let $ab\in {\rm ann}_R(Z(R))$, $a\not\in {\rm ann}_R(Z(R)) $ and $b\not\in {\rm ann}_R(Z(R))$. Since $|\sum|= 2$, we conclude that $ab\neq 0$, $b^2\neq 0$ and ${\rm ann}_R(b)={\rm ann}_R(b^2)$. If $ab=0$, then $|\sum\neq 2$, a contradiction. Thus $ab\neq0$. Since $ab\in {\rm ann}_R(Z(R))$, $abb=0$. This, together with ${\rm ann}_R(b)={\rm ann}_R(b^2)$,  implies that $ab=0$, a contradiction. Hence $\sum=\mathrm {Ass}(R)$.

$(4)\Rightarrow (1)$ is clear.
}
\end{proof}

We end this paper with the following example.

\begin{example}
\end{example}

Let $R= \mathbb{Z}_{2}[X,Y]/(XY, X^2) $, and let $x=X+(XY+X^2)$ and $y=Y+(XY+X^2)$. Then $Z(R)=(x,y)R$ and $\mathrm{Nil}(R)=\{0,x\}$. Then the following statements hold.

$(1)$ ${\rm ann}_R(Z(R))$ is a prime ideal of $R$.

$(2)$  $\Gamma^{^{\prime}}(R)=K_{|Z(R)^*|}$ and $ \overline{AG(R)}=\overline{K}_{|{\rm ann}_R(Z(R))^*|}+K_{|Z(R)\setminus {\rm ann}_R(Z(R))|}$.

$(3)$ ${\rm ann}_R(Z(R))=\mathrm{Nil}(R)$ and  $|Z(R)|=\infty$.

$(4)$ $R$ is a non-reduced ring and $\sum=\mathrm {Ass}(R)$.

{\section{Conclusion}}\vspace{-2mm}
In this paper, the authors introduced the inclusion graph of annihilators and studied various inter-relationships between $\Gamma^{^{\prime}}(R)$ as a
graph and $R$ as a ring. The main goal of these discussions is to study connectedness, diameter, girth, completeness and affinity between $\Gamma^{^{\prime}}(R)$ and $\overline{AG(R)}$. As
a topic of further research, one can look into the clique and chromatic numbers of such graphs and the structures of rings $R$ with $\mathrm{diam}(\Gamma^{^{\prime}}(R))\in \{3, 4\}$.

%%%%%%%%%%%%%%%%%%%%%%%%%%%%%%%%%%%%%%%%%%%%%%%%%%%%%%%%%%%%%%%%%%%%%%%%%%%%%%%%%%%%%%%%%%%%%%%%%%%%%%%%%%%%%%%%%%%%%%%%%%%%%%%%%%%%%%%%%%%%%%%%%%%%%%
{}

\end{document}